\documentclass[11pt]{article}
\usepackage{amsmath, amssymb, amsthm}
\usepackage{verbatim}
\usepackage{multicol}
\usepackage{enumerate}
\usepackage{comment}
\usepackage[none]{hyphenat}
\usepackage{hyperref}
\hypersetup{
	colorlinks=true,
	linkcolor=blue,
	filecolor=maghttps://www.overleaf.com/projectenta,      
	urlcolor=cyan,
	citecolor=blue
}
\usepackage{pgf}
\usepackage{tikz}
\usetikzlibrary{positioning,arrows,shapes,decorations.markings,decorations.pathreplacing,matrix,patterns}
\tikzstyle{vertex}=[circle,draw=black,fill=black,inner sep=0,minimum size=3pt,text=white,font=\footnotesize]
\usepackage{cleveref}

\date{}
\title{\vspace{-1cm} Flattening rank and its combinatorial applications}

\author{
David Munh\'a Correia\thanks{ETH Zurich, \emph{e-mail}: \textbf{\{david.munhacanascorreia,benjamin.sudakov,istvan.tomon\}@math.ethz.ch}.}
\and
Benny Sudakov\footnotemark[1]
\and
Istv\'an Tomon\footnotemark[1]
}

\theoremstyle{plain}
\newtheorem{theorem}{Theorem}

\newtheorem{claim}[theorem]{Claim}

\Crefname{theorem}{Theorem}{Theorems}
\Crefname{definition}{Definition}{Definitions}
\Crefname{corollary}{Corollary}{Corollaries}
\Crefname{claim}{Claim}{Claims}
\Crefname{lemma}{Lemma}{Lemmas}
\Crefname{conjecture}{Conjecture}{Conjectures}
\Crefname{problem}{Problem}{Problems}
\Crefname{prop}{Proposition}{Propositions}

\theoremstyle{definition}

\DeclareMathOperator{\frank}{frank}
\DeclareMathOperator{\mfrank}{mfrank}

\DeclareMathOperator{\supp}{supp}

\newcommand{\m}{\mathbf}

\begin{document}

\maketitle
\sloppy

\begin{abstract}
Given a $d$-dimensional tensor $T:A_1\times\dots\times A_d\rightarrow \mathbb{F}$ (where $\mathbb{F}$ is a field), the \emph{$i$-flattening rank} of $T$ is the rank of the matrix whose
rows are indexed by $A_{i}$, columns are indexed by $B_{i}=A_1\times\dots\times A_{i-1}\times A_{i+1}\times\dots\times A_{d}$ and whose entries are given by the corresponding values of $T$.
The \emph{max-flattening rank} of $T$ is defined as $\mfrank(T)=\max_{i\in [d]}\frank_{i}(T)$. A tensor $T:A^{d}\rightarrow\mathbb{F}$ is called semi-diagonal, if $T(a,\dots,a)\neq 0$ for every $a\in A$, and $T(a_{1},\dots,a_{d})=0$ for every $a_{1},\dots,a_{d}\in A$ that are all distinct. In this paper we prove that if $T:A^{d}\rightarrow\mathbb{F}$ is semi-diagonal, then $\mfrank(T)\geq \frac{|A|}{d-1}$, and this bound is the best possible.
    
We give several applications of this result, including a generalization of the celebrated Frankl-Wilson theorem on forbidden intersections. Also,  addressing a conjecture of Aharoni and Berger, we show that if the edges of an $r$-uniform multi-hypergraph $\mathcal{H}$ are colored with $z$ colors such that each colorclass is a matching of size $t$, then $\mathcal{H}$ contains a rainbow matching of size $t$ provided $z>(t-1)\binom{rt}{r}$. This improves previous results of Alon and Glebov, Sudakov and Szab\'o.
\end{abstract}

\section{Introduction}

A \emph{$d$-dimensional tensor} over a field $\mathbb{F}$ is a function $T:A_1\times\dots\times A_d\rightarrow \mathbb{F}$, where $A_1,\dots,A_d$ are finite sets. For $i\in [d]$, the \emph{$i$-flattening rank} of $T$, denoted by $\frank_{i}(T)$, is defined as follows. Let $B_{i}=A_1\times\dots\times A_{i-1}\times A_{i+1}\times\dots\times A_{d}$, and view $T$ as a matrix $M$ with rows indexed by $A_{i}$, and columns indexed by $B_{i}$. Then $\frank_{i}(T):=\mbox{rank}(M)$. Note that $\frank_{i}(T)=1$ if and only if $T\neq 0$, and there exist two functions $f:A_{i}\rightarrow \mathbb{F}$ and $g:B_i\rightarrow \mathbb{F}$ such that $T(a_1,\dots,a_d)=f(a_i)g(a_1,\dots,a_{i-1},a_{i+1},\dots,a_{d})$. Also, the $i$-flattening rank of $T$ is the minimum $r$ such that $T$ is the sum of $r$ tensors of $i$-flattening rank 1. Equivalently, $\frank_{i}(T)$ is the dimension of the vector space generated by the rows of $T$ in the $i$-th dimension. Define the \emph{max-flattening rank} of $T$ as $$\mfrank(T)=\max_{i\in [d]}\frank_{i}(T).$$

It is easy to see that the max-flattening rank and $i$-flattening rank satisfy the usual properties of rank. More precisely, they are subadditive, and if $T'$ is a subtensor of $T$, then $\frank_i(T')\leq \frank_i(T)$ and $\mfrank(T')\leq \mfrank(T)$. 
Here, $T':A_1'\times\dots\times A_{r}'\rightarrow\mathbb{F}$ is a \emph{subtensor} of $T:A_1\times\dots\times A_{r}\rightarrow\mathbb{F}$ if $A_{i}'\subset A_{i}$ for $i\in [r]$, and $T'(a_1,\dots,a_r)=T(a_1,\dots,a_r)$ for $(a_1,\dots,a_{r})\in A_{1}'\times\dots\times A_{r}'$.
Also, the usual notion of tensor rank is always an upper bound for the max-flattening rank. As a reminder, $T$ has \emph{tensor rank} 1 if there are $d$ function $f_1,\dots,f_d$ such that $T(a_1,\dots,a_d)=f_1(a_1)\dots f_d(a_d)$, and the tensor rank 
$trank(T)$ is the minimal $r$ such that $T$ is the sum of $r$ tensors of tensor rank 1.

In this paper, we are interested in combinatorial applications of the max-flattening rank. Note that one of the trivial, but important property of the matrix rank is that diagonal matrices have full rank. The analogue of this is also trivially true for the flattening rank: if $T:A^{d}\rightarrow \mathbb{F}$ is a diagonal tensor, that is, $T(a_{1},\dots,a_{d})\neq 0$ if and only if  $a_1=\dots=a_{d}$, then $\frank_{i}(T)=\mfrank(T)=|A|$ for $i\in [d]$. However, in certain applications this is not really what is needed, thus we would like to relax the notion of diagonality. 

Say that $d$-dimensional tensor $T:A^{d}\rightarrow \mathbb{F}$ is \emph{semi-diagonal} if the following holds: $T(a_1,\dots,a_{d})=0$ if $a_1,\dots,a_d$ are all distinct, and $T(a_1,\dots,a_{d})\neq 0$ if $a_1=\dots=a_d$. If $a_1,\dots,a_d$ are neither all equal or all distinct, then there is no restriction on $T(a_1,\dots,a_d)$. Our main technical result is the following lower bound on the rank of semi-diagonal tensors.

\begin{theorem}\label{thm:main}
Let $T:A^{d}\rightarrow\mathbb{F}$ be a $d$-dimensional semi-diagonal tensor. Then $$\mfrank(T)\geq \frac{|A|}{d-1}.$$
\end{theorem}

Let us make a few remarks about this theorem.  The bound $\mfrank(T)\geq \left\lceil\frac{|A|}{d-1}\right\rceil$ is the best possible for any positive integers $d\geq 2$ and $|A|$. Indeed, let $A_1,\dots,A_m$ be partition of $A$ into $m=\left\lceil\frac{|A|}{d-1}\right\rceil$ parts of size at most $d-1$, and define the tensor $T:A^d\rightarrow \mathbb{F}$ such that $$T(a_1,\dots,a_d)=\begin{cases} 1 &\mbox{if }a_1,\dots,a_d\in A_{i}\mbox{ for some }i\in [m]\\
    0 &\mbox{otherwise.}\end{cases}$$
Then $T$ is semi-diagonal, and the $i$-flattening rank of $T$ is exactly $m$ for $i\in [d]$. Moreover, the $i$-flattening rank of a semi-diagonal tensor need not be large for any fixed $i$. Indeed, if $T:A^{d}\rightarrow \mathbb{F}$ is defined as $$T(a_1,\dots,a_d)=\begin{cases} 1 &\mbox{if }a_1=\dots=a_{i-1}=a_{i+1}=\dots=a_{d}\\
0 &\mbox{otherwise.}\end{cases}$$
then $\frank_{i}(T)=1$. 

We will prove Theorem \ref{thm:main} in the next section and provide some of its combinatorial applications in Section \ref{sect:applications}. We also like to mention that further application of Theorem \ref{thm:main} appears in \cite{ST20}, where it is used  to establish certain Ramsey properties of algebraic hypergraphs. 

\section{Semi-diagonal tensors}

In this section, we prove Theorem \ref{thm:main}. More precisely, we prove the following theorem, which then immediately implies Theorem \ref{thm:main}.

\begin{theorem}
Let $T:A^{d}\rightarrow\mathbb{F}$ be a semi-diagonal tensor. Then $$\sum_{i=1}^{d}\frank_{i}(T)\geq \frac{d}{d-1}|A|.$$
\end{theorem}

\begin{proof}
Let us introduce some notation. If $\m{v}\in \mathbb{F}^{A}$, let $\supp(\m{v})$ be the \emph{support} of $\m{v}$, that is, the set of elements $b\in A$ such that $\m{v}(b)\neq 0$. If $\m{a}\in A^{d}$, let $\{\m{a}\}=\{\m{a}(1),\dots,\m{a}(d)\}\subset A$ be the set of elements of $A$ which appear as coordinates of vector $\m{a}$. Let $\phi(\m{a})\subset [d]$ be the set of indices $i\in [d]$ such that $\m{a}(i)$ appears at least twice among $\m{a}(1),\dots,\m{a}(d)$. Also, for $i\in [d]$, let $\m{a}[i]\in \mathbb{F}^{A}$ be the vector defined as $$\m{a}[i](b)=T(\m{a}(1),\dots,\m{a}(i-1),b,\m{a}(i+1),\dots,\m{a}(d))$$
for $b\in A$. Finally, let $U_{i}(T)=\{\m{a}[i]:\m{a}\in A^{d}\}$, and let $V_{i}(T)$ be the subspace of $\mathbb{F}^{A}$ generated by the elements of $U_{i}(T)$. Then, by definition,  $$\frank_{i}(T)=\mbox{dim}(V_{i}(T)).$$

We prove the theorem by induction on $|A|$. If $|A|\leq d-1$, the statement is clearly true as the $i$-flattening is at least $1$ for $i\in [d]$, so let us assume that $|A|\geq d$. Choose $\m{a}\in A^{d}$ such that $T(\m{a})\neq 0$, and the set $\{\m{a}\}$ has maximal size. Then $|\{\m{a}\}|\leq d-1$ as $T$ is semi-diagonal.

\begin{claim}\label{claim:supp1}
If $i\in \phi(\m{a})$, then $\supp(\m{a}[i])\subset \{\m{a}\}$.
\end{claim}
\begin{proof}
Suppose this is not the case, and let $c\in \supp(\m{a}[i])\setminus \{\m{a}\}$. Let $\m{a'}$ be the $d$-tuple we get by replacing $\m{a}(i)$ with $c$ in $\m{a}$. Then $T(\m{a'})\neq 0$ and $\{\m{a'}\}=\{\m{a}\}\cup\{c\}$, contradicting the maximality of $|\{\m{a}\}|$.
\end{proof}

Let $\mathcal{A}$ be the set of all $d$-tuples $\m{b}\in A^d$ such that $T(\m{b})\neq 0$ and $\{\m{b}\}=\{\m{a}\}$. Define the graph $G$ on $\mathcal{A}$ as follows: connect $\m{b}$ and $\m{b'}$ by an edge if they differ in exactly one coordinate. Let $\mathcal{C}\subset \mathcal{A}$ be the connected component of $G$ containing $\m{a}$. Say that an index $j\in [d]$ is \emph{good} if there exists $\m{b}\in \mathcal{C}$ such that $j\in \phi(\m{b})$, and let $J\subset [d]$ be the set of good indices. 

\begin{claim}\label{claim:fix}
If $j$ is not good, then $\m{a}(j)=\m{b}(j)$ for every $\m{b}\in\mathcal{C}$.
\end{claim}
\begin{proof}
As $\m{a}$ and $\m{b}$ are in the same connected component, there exists a path from $\m{a}$ to $\m{b}$ in $G$, which means that there exists a sequence $\m{a}=\m{a}_0,\dots,\m{a}_p=\m{b}$ of elements of $\mathcal{A}$ such that $\m{a}_{k}$ and $\m{a}_{k+1}$ differ in exactly one coordinate for $k=0,\dots,p-1$, say in coordinate $j_{k}$. But as $\{\m{a}_{k}\}=\{\m{a}_{k+1}\}$, we have that $\m{a}_{k}(j_{k})$ is equal to some other coordinate $\m{a}_{k}(j')$. Therefore, $j_{k}$ is good, and so the coordinate at $j$ was never changed.
\end{proof}

Let $X=\{\m{a}(j):j\in [d]\setminus J\}$. By definition of goodness, all $\m{a}(j)$ are distinct elements of $A$ and appear only once as coordinate of $\m{a}$. Thus
$|X|=d-|J|$. Also, let $Y=\{\m{a}\}\setminus X$, then 
$$1\leq |Y|=|\{\m{a}\}|-|X|\leq |J|-1.$$ 
For every $j\in J$, pick an element $\m{b}\in \mathcal{C}\subset \mathcal{A}$ such that $j\in\phi(\m{b})$, and define the vector $\m{v}_{j}=\m{b}[j]\in V_{j}(T)$. 

\begin{claim}
$\supp(\m{v}_{j})\subset Y$.
\end{claim}
\begin{proof}
As $j\in\phi(\m{b})$, we have by Claim \ref{claim:supp1} that $\supp(\m{v}_j)\subset \{\m{b}\}=\{\m{a}\}$. Since $Y=\{\m{a}\}\setminus X=\{\m{b}\}\setminus X$, if $\supp(\m{v}_{j})\not\subset Y$, then there exists $c\in X$ such that $c\in \supp(\m{v}_j)$. Then $c\neq \m{b}(j)$. Let $\m{b'}$ be the $d$-tuple we get after replacing $\m{b}(j)$ with $c$ in $\m{b}$. Since $c \in \{\m{b}\}$, we have that $\{\m{b'}\}=\{\m{b}\}=\{\m{a}\}$ and $\m{b'}$ differs from $\m{b}$ in exactly one coordinate.
Then $\m{b'}\in\mathcal{C}$ but $\m{b'}(j)\neq \m{b}(j)$, contradicting Claim \ref{claim:fix}.
\end{proof}

Let $A'=A\setminus Y$ and let $T'$ be the restriction of $T$ to $(A')^{d}$. Then $T'$ is an $(|A|-|Y|)$-sized semi-diagonal tensor, so by our induction hypothesis, we have $$\sum_{i=1}^{d}\mbox{dim}(V_{i}(T'))\geq \frac{d}{d-1}(|A|-|Y|).$$
However, note that for $j\in J$, the support of the vector $\m{v}_j$ is disjoint from $A'$. Let $\m{e}_1,\dots,\m{e}_r\in (A')^d$ such that the restriction of the vectors $\m{e}_1[j],\dots,\m{e}_r[j]$ to $A'$ is a basis of $V_{j}(T')$ (so $r=\mbox{dim}(V_j(T'))$), then the vectors $\m{v}_{j},\m{e}_{1}[j],\dots,\m{e}_{r}[j]$ are linearly independent in $\mathbb{F}^{A}$. Therefore, $$\mbox{dim}(V_j(T))\geq \mbox{dim}(V_j(T'))+1.$$ 
But then
$$\sum_{i=1}^{d}\mbox{dim}(V_{i}(T))\geq |J|+\sum_{i=1}^{d}\mbox{dim}(V_{i}(T'))\geq |J|+\frac{d}{d-1}(|A|-|Y|)\geq \frac{d}{d-1}|A|,$$
where the last inequality holds noting that $|Y|\leq |J|-1\leq d-1$.

\end{proof}

\section{Applications}\label{sect:applications}

\subsection{Oddtown}

A family $\mathcal{A}$ of subsets of an $n$-element set is an \emph{Oddtown} if $|A|$ is odd for every $A\in\mathcal{A}$, and $|A\cap B|$ is even for every $A,B\in\mathcal{A}, A\neq B$. It was proved famously by Berkelamp \cite{B69} that the size of an Oddtown on an $n$ element ground set is at most $n$.

The following generalization of this problem was considered by Vu \cite{V99}. A \emph{$d$-wise Oddtown} is
a family $\mathcal{A}$ of subsets of some ground set such that $|A|$ is odd for every $A\in\mathcal{A}$, and $|A_1\cap \dots\cap A_{d}|$ is even for every $A_1,\dots,A_{d}\in\mathcal{A}$ that are all distinct. Vu \cite{V99} proved that the size of a $d$-wise Oddtown on an $n$ element ground set is at most $(d-1)n$. In the case we allow repetitions in $\mathcal{A}$ this bound is also the best possible.

As our first application, we show that the cross-version of this result also holds with the same bound. A \emph{cross-$d$-wise Oddtown} is a family $\mathcal{A}$ such that  
every $A\in\mathcal{A}$ is an ordered $d$-tuple $A(1), \ldots, A(d)$ of subsets of the ground set (which do not need to be distinct) with the property that $|A(1)\cap\dots\cap A(d)|$ is odd for 
every $A\in\mathcal{A}$, and $|A_1(1)\cap \dots\cap A_{d}(d)|$ is even for every $A_1,\dots,A_{d}\in\mathcal{A}$ that are all distinct. Note that $d$-wise Oddtown is a special case of cross-$d$-wise Oddtown, where every $d$-tuple 
contains the same set $d$ times.

\begin{theorem}
If $\mathcal{A}$ is a cross-$d$-wise Oddtown on an $n$ element ground set, then $|\mathcal{A}|\leq (d-1)n$.
\end{theorem}

\begin{proof}
Let $m=|\mathcal{A}|$. Define the tensor $T:\mathcal{A}^{d}\rightarrow \mathbb{F}_2$ such that for $A_{1},\dots,A_{d}\in\mathcal{A}$ we have $T(A_1,\dots,A_d)=|A_{1}(1)\cap\dots\cap A_{d}(d)|$. Then $\mathcal{A}$ is a semi-diagonal tensor and we have $\mfrank(T)\geq \frac{m}{d-1}$, by Theorem \ref{thm:main}. On the other hand, we show that $\mbox{trank}(T)\leq n$, which then implies $\mfrank(T)\leq n$. For $A\in\mathcal{A}$ and $j\in [d]$, let $v_{A(j)}:[n]\rightarrow \mathbb{F}_2$ be the characteristic function of $A(j)$. For $k\in [n]$, define the function $f_{j,k}:\mathcal{A}\rightarrow \mathbb{F}_2$ as $f_{j,k}(A)=v_{A(j)}(k)$. Then 
$$T(A_1,\dots,A_{d})=\sum_{k=1}^{n}f_{1,k}(A_1)\dots f_{d,k}(A_d).$$
Therefore, $\mbox{trank}(T)\leq n$.
\end{proof}

\subsection{Forbidden intersections}
Let $p$ be a prime, $L\subset \mathbb{F}_{p}$ be a set of residues $\hspace{-0.25cm} \mod p$, and $\mathcal{F}\subset 2^{[n]}$ be a family such that $|A|\not\in L$ for every $A\in\mathcal{F}$, but $|A\cap B|\in L$ for every distinct $A,B\in\mathcal{F}$. The celebrated Frankl-Wilson theorem \cite{FW81} on forbidden intersections says that
$$|\mathcal{F}|\leq \sum_{s=0}^{|L|}\binom{n}{s}.$$ 

A natural extension of this was given by Grolmusz and Sudakov \cite{GS02}, who proved if one requires that all $|A|\not\in L$ but every intersection of $k$ distinct sets in $\mathcal{F}$ has size in $L$, then $|\mathcal{F}|\leq (k-1)\sum_{s=0}^{|L|}\binom{n}{s}$. This bound is tight if we allow $\mathcal{F}$ to be a multiset. Here, we prove the following extension of these results for more general, forbidden configurations. A \emph{configuration of order $k$ modulo $p$} is a pair $(\mathcal{C},L)$, where $\mathcal{C}\subset 2^{[k]}$ and $L\subset \mathbb{F}_p$.  Say that a family $\mathcal{F}\subset 2^{[n]}$ is \emph{$(\mathcal{C},L)$-satisfying}, if $|A|\not\in L$ for every $A\in\mathcal{F}$, but there exist no $k$ distinct sets $A_{1},\dots,A_{k}\in\mathcal{F}$ such that $|\bigcap_{i\in X}A_{i}|\not\in L$ for every $X\in \mathcal{C}$.

Clearly, asking $\mathcal{F}$ to be $(\{\{1,2\}\},L)$-satisfying is equivalent to the condition of the Frankl-Wilson theorem and being  $(\{\{1, \ldots, k\}\},L)$-satisfying is equivalent to restricted $k$-wise intersections. We bound the size of the  maximal $(\mathcal{C},L)$-satisfying family by a function of $n$, $|L|$ and the maximum degree of $\mathcal{C}$, which we define next. 
Given a family $\mathcal{C}$,
the \emph{degree} of $a\in [k]$ in $\mathcal{C}$ is the number of sets in $\mathcal{C}$ containing $a$, and is denoted by $\mbox{deg}_{\mathcal{C}}(a)$. The \emph{maximum degree} of $\mathcal{C}$ is $\Delta(\mathcal{C})=\max_{a\in [k]}\mbox{deg}_{\mathcal{C}}(a)$.

\begin{theorem}\label{thm:fw}
    Let $(\mathcal{C},L)$ be a configuration of order $k$ modulo $p$, and let $\Delta=\Delta(\mathcal{C})$. If $\mathcal{F}\subset 2^{[n]}$ is $(\mathcal{C},L)$-satisfying, then $$|\mathcal{F}|\leq (k-1)\sum_{s=0}^{\Delta|L|}\binom{n}{s}.$$
\end{theorem}

\begin{proof}
     Let $h:\mathbb{F}_{p}\rightarrow\mathbb{F}_{p}$ be the polynomial defined as $h(x)=\prod_{\ell\in L}(x-\ell)$. Define the $k$-dimensional tensor $T:\mathcal{F}^{k}\rightarrow \mathbb{F}_{p}$ as follows. For $A_1,\dots,A_{k}$, let $$T(A_1,\dots,A_{k})=\prod_{X\in \mathcal{C}}h\left(\left|\bigcap_{i\in X}A_{i}\right|\right).$$
     Then $T$ is semi-diagonal as $\mathcal{F}$ is $(\mathcal{C},L)$-satisfying. Therefore, by Theorem \ref{thm:main}, we have $$\mfrank(T)\geq \frac{|\mathcal{F}|}{k-1}.$$
     
     We show that for $j\in [k]$, the $j$-flattening rank of $T$ is at most $|\mathcal{F}|\leq \sum_{s=0}^{d_j}\binom{n}{s},$ where $d_{j}=\mbox{deg}_{\mathcal{C}}(j)|L|$. For ease of notation, let us show this for $j=1$, it follows for the other values of $j$ by the same reasoning.
     
     For $A\in \mathcal{F}$, let $v_{A}\in \mathbb{F}_p^{n}$ be the characteristic vector of $A$. Let $A_1,\dots,A_k\in\mathcal{F}$, then 
     $$T(A_{1},\dots,A_{k})=\prod_{X\in\mathcal{C}}h\left(\sum_{j=1}^{n}\prod_{i\in X}v_{A_{i}}(j)\right).$$
     Let $p:\mathbb{F}_p^{kn}\rightarrow \mathbb{F}_{p}$ be the polynomial defined as 
     $$p(v_{1},\dots,v_{k})=\prod_{X\in\mathcal{C}}h\left(\sum_{j=1}^{n}\prod_{i\in X}v_{i}(j)\right),$$
     where $v_1,\dots,v_{k}\in \mathbb{F}^{n}$. 
    Write $p$ as the sum of monomials and in each monomial replace $v_i(j)^a, a \geq 1$ by $v_i(j)$.
    Let $q$ be the resulting polynomial and note that $q(v_1,\dots,v_{k})=p(v_1,\dots,v_k)$ if  $v_{1},\dots,v_{k}$ are characteristic vectors, having all their coordinates equal $0$ or $1$. 
     
     The polynomial $q(v_1,\dots,v_k)$ can be written as the sum of polynomials of the form $$\beta_{J}\left(\prod_{j=1}^{n}v_{1}(j)^{J(j)}\right)q_{J}(v_2,\dots,v_{k}),$$ where $J\in \{0,1\}^{n}$, $\beta_{J}\in\mathbb{F}_p$ and $q_{J}:\mathbb{F}_p^{(k-1)n}\rightarrow\mathbb{F}_p$ is some polynomial. Note that $\beta_{J}=0$ unless $|J|\leq |L|\mbox{deg}_{\mathcal{C}}(1)$. Let $\mathcal{J}=\{J\in\{0,1\}^{n}:|J|\leq |L|\mbox{deg}_{\mathcal{C}}(1)\}$. For $J\in\mathcal{J}$, define the functions $f_{J}:\mathcal{F}\rightarrow\mathbb{F}_{p}$ and $g_{J}:\mathcal{F}^{k-1}\rightarrow\mathbb{F}_{p}$ as
     $$f_{J}(A_{1})=\beta_{J}\prod_{j=1}^{k}v_{A_1}(j)^{J(j)},$$
     and
     $$g_{J}(A_2,\dots,A_{k})=q_{J}(v_{A_2},\dots,v_{A_k}).$$
     Then 
     $$T(A_1,\dots,A_{k})=\sum_{J\in\mathcal{J}}f_{J}(A_1)g_{J}(A_2,\dots,A_k),$$
     which proves that $\frank_{1}(T)\leq |\mathcal{J}|\leq \sum_{s=0}^{d_{1}}\binom{n}{s}$. As the corresponding bound holds for the $j$-flattening as well for $j\in [k]$, we get 
     $$\mfrank(T)\leq \sum_{s=0}^{\Delta(\mathcal{C})|L|}\binom{n}{s}.$$
     Comparing the lower and upper bound on the max-flattening rank, we get the desired bound $$|\mathcal{F}|\leq (k-1)\sum_{s=0}^{\Delta(\mathcal{C})|L|}\binom{n}{s}.$$
\end{proof}

So far we showed that for a fixed configuration $(\mathcal{C},L)$ of order $k$ modulo $p$, the maximal size of a $(\mathcal{C},L)$-satisfying family of subsets of $[n]$ is of order at most $n^{\Delta |L|}$, where $\Delta$ is the maximum degree of the set family $\mathcal{C}$. For $\Delta=1$ the exponent of $n$ is clearly best possible, as the Frankl-Wilson bound is sharp. On the other hand, for $\Delta\geq 2$ we do not know how accurate our result is.
Nevertheless, we can show that the exponent of $n$ must depend on $\Delta$. 

Indeed, consider the case $p = 2$, $\mathcal{C}_k=[k]^{(2)}$ is a complete graph of order $k$ and $L=\{0\}$. Then $\Delta=k-1$ and we construct families in $2^{[n]}$ which are $(\mathcal{C}_k,\{0\})$-satisfying and have size $n^{\Omega(\log k/\log\log k)}$. Our construction is a modification of an argument of Alon and Szegedy \cite{ASz99}.

\begin{theorem}
Let $t,s$ be positive integers, $k = \lfloor \frac{2^{t+1}}{t-1} \rfloor$, and $n = t^s$. Then, there exists a family $\mathcal{F} \subseteq 2^{[n]}$ which is $(\mathcal{C}_k,\{0\})$-satisfying and has size at least $2^{(t-1)s/4}$.
\end{theorem}
\begin{proof}
Let $\mathcal{G} \subseteq 2^{[t]}$ be a family of odd-sized sets all whose pairwise intersections have also odd size. By the well known variation of the Oddtown problem
we have that $|\mathcal{G}| \leq 2^{(t-1)/2}$. 

Now, let $\mathcal{O}_t \subseteq 2^{[t]}$ denote the family of all odd-sized subsets of $[t]$, which clearly has size $2^{t-1}$. Since $n = t^s$, we can identify $[n]$ with the set $[t]^s$. Let then $\mathcal{O}^s_t \subseteq 2^{[n]}$ denote the family of sets of the form $A_1 \times \ldots \times A_s$, where $A_i \in \mathcal{O}_t$ for all $i$. Note that all sets in $\mathcal{O}^s_t$ are odd-sized. 
If $\mathcal{F}_1, \ldots, \mathcal{F}_s$ are subsets of $\mathcal{O}_t$ with none of $\mathcal{F}_i$ containing a pair of sets with even-sized intersection, then we call the collection 
$$\mathcal{F}_1 \times \ldots \times \mathcal{F}_s := \{A_1 \times \ldots \times A_s : A_i \in \mathcal{F}_i\} \subseteq \mathcal{O}^s_t$$ 
a \emph{bad box}. As we explained above, in this case $|\mathcal{F}_i|\leq 2^{(t-1)/2}$, and therefore every bad box contains at most $2^{s(t-1)/2}$ sets. 
Note that the intersection of $A_1\times\dots \times A_s$ and $B_1\times\dots\times B_s$ is $(A_1\cap B_1)\times\dots\times (A_s\cap B_s)$. 
Therefore, $\mathcal{F}_1\times\dots\times \mathcal{F}_s$ is a bad box if and only if it contains no two sets with even intersection.

Take $\mathcal{F}$ to be a random family given by choosing uniformly and independently, with repetition, $\lceil 2^{(t-1)s/4} \rceil$ sets in $\mathcal{O}^s_t$. Then, for every bad box $\mathcal{B}$, the probability that at least $k$ elements of $\mathcal{F}$ are contained in $\mathcal{B}$ is at most $\binom{|\mathcal{F}|}{k}\left(\frac{|\mathcal{B}|}{|\mathcal{O}^s_t|} \right)^k \leq  2^{-(t-1)sk/4}$. The number of bad boxes can be upper bounded by $2^{s2^{t-1}}$. Thus, using
our choices of $k$ and $t$, we conclude that the probability that some bad box contains $k$ elements of $\mathcal{F}$ is at most
     $$2^{s2^{t-1}}  \cdot 2^{-(t-1)sk/4} \leq 1\,.$$
To finish, note this implies that $\mathcal{F} \subseteq 2^{[n]} $ is $(\mathcal{C}_k, \{0\})$-satisfying. Indeed, suppose this is not the case. Since each member of $\mathcal{O}^s_t$ has odd size, there exist $k$ sets $S_1, \ldots, S_k \in \mathcal{F}$ such that for all distinct $p,q$, we have $|S_p \cap S_q| = 1 \text{ (mod 2)}$. Let the collection $\{A^{(i)}_j: 1 \leq i \leq k, 1 \leq j \leq s \}$ be such that $S_i = A^{(i)}_1 \times \ldots \times A^{(i)}_s$ for all $i$ and for each $j$, let $\mathcal{F}_j = \{A^{(i)}_j: 1 \leq i \leq k \} \subseteq \mathcal{O}_t$. Since the size of $S_p \cap S_q$ is the product of the sizes of the intersections $A^{(p)}_j \cap A^{(q)}_j$, we must have that each collection $\mathcal{F}_j$ has only odd-sized pairwise intersections. Hence, $\mathcal{F}_1 \times \ldots \times \mathcal{F}_s$ is a bad box and contains the sets $S_1, \ldots, S_k$. This is a contradiction, since no bad box contains $k$ members of $\mathcal{F}$.
\end{proof}

\noindent 
As $\Delta=k-1$ and $t=\Theta( \log k)$, indeed, the family provided by the previous theorem has size $n^{\Omega(\log \Delta/\log\log\Delta)}$. It would be interesting to improve this result and get a better understanding of how much the exponent of $n$ should depend on $\Delta$.

\subsection{Rainbow matchings}

Let $\mathcal{H}$ be an $r$-uniform multi-hypergraph (that is, we allow repetitions of the edges). Given a coloring $c:E(\mathcal{H})\rightarrow [z]$, a rainbow matching in $\mathcal{H}$ is a matching in which no two edges have the same color. The hypergraph $\mathcal{H}$ is \emph{$(z,t)$-colored} if it is colored with $z$ colors, and each colorclass is a matching of size $t$. Let $f(r,t)$ denote the maximal $z$ such that there exists a $(z,t)$-colored $r$-partite $r$-uniform multi-hypergraph which contains no rainbow matching of size $t$. Also, let $F(r,t)$ denote the maximal $z$ such that there exists a $(z,t)$-colored $r$-uniform multi-hypergraph which contains no rainbow matching of size $t$. Clearly, $f(r,t)\leq F(r,t)$. Aharoni and Berger \cite{AB09} proved that $f(r,t)\geq (t-1)2^{r}$, and equality holds if $r=2$ or $t=2$. They also conjectured that $f(r,t)=(t-1)2^{r}$ holds in general. This was disproved by Alon \cite{A11}, who showed that $f(r,3)\geq 2.71^{r}$. More precisely, Alon discovered a connection between $f(r,t)$ and the following well studied function. Let $g(r,t)$ denote the smallest integer $g$ such that any sequence of $g$ elements of the Abelian group $\mathbb{Z}_{t}^{r}$ contains a subsequence of length $t$, whose elements sum up to zero. Then  $f(r,t)\geq g(r-1,t)-1$.

On the other hand, Glebov, Sudakov and Szab\'o \cite{GSSz14} proved, using combinatorial techniques, that $$F(r,t)\leq \min\{(r+1)^{2r+1}t^{2r+1},8^{rt}\}.$$
Our next theorem improves this upper bound for every $(r,t)$ satisfying $r,t\geq 3$, which also improves all known upper bounds for $f(r,t)$ as well.
 
 \begin{theorem}\label{thm:rainbow}
 	$F(r,t)\leq (t-1)\binom{rt}{r}.$
 \end{theorem}

The proof is based on the exterior algebra method. The interested reader can find a detailed description of this method as well as various applications in Chapter 6 of the book by Babai and Frankl \cite{BF88}. Here, let us  only give a basic introduction to exterior algebras. 

Let $V$ be a vector space over some field $\mathbb{F}$. The \emph{exterior algebra $\bigwedge V$} is the associative algebra generated by the elements of $V$ and the associative binary operation $\wedge$, called \emph{wedge product (or exterior product)}. Subject to these, $\wedge$ has the additional property that $v\wedge v=0$ for all $v\in V$. Also, the \emph{$k$-th exterior power of $k$}, denoted by $\wedge^{k} V$, is the vector space generated by the elements $v_1\wedge\dots\wedge v_k$, where $v_1,\dots,v_k\in V$. Let us list some of the well known properties of the wedge product, $\bigwedge^k V$ and $\bigwedge V$.

\begin{enumerate}
	\item ($\bigwedge V$ is an associative algebra.) If $a,b,c\in \bigwedge V$ and $\lambda\in \mathbb{F}$, then \begin{align*}&(a\wedge b)\wedge c=a\wedge(b\wedge c),\\
	&a\wedge (\lambda b)=(\lambda a)\wedge b=\lambda(a\wedge b),\\
	&a\wedge (b+c)=(a\wedge b)+(a\wedge c),\\
	&(a+b)\wedge c=(a\wedge c)+(b\wedge c).
	\end{align*}
    \item If $v,w\in V$, then $v\wedge w=-w\wedge v$.
    \item If $v_1,\dots,v_k$, then $v_1\wedge\dots\wedge v_k\neq 0$ if and only if $v_1,\dots,v_k$ are linearly independent.
    \item If $\mbox{dim}(V)=n$, then $\mbox{dim}(\bigwedge^{k}V)=\binom{n}{k}.$ Moreover, if $e_1,\dots,e_n$ is a basis of $V$, then $$\{e_{i_1}\wedge\dots\wedge e_{i_k}\}_{1\leq i_1<\dots<i_k\leq n}$$ is a basis of $\bigwedge^k V$.
\end{enumerate}

If $\mathbb{F}$ has characteristic 2, then $\wedge$ is also commutative by property 2. Therefore, in this case, if $A=\{a_1,\dots,a_k\}\subset V$, we can write $\bigwedge_{a\in A}a$ instead of $a_1\wedge\dots\wedge a_k$ without specifying the order of terms.

\begin{proof}[Proof of Theorem \ref{thm:rainbow}]
	 Let $n=rt$. Let $\mathcal{H}$ be an $r$-uniform multi-hypergraph with $(z,t)$-coloring $c$. Let $\mathbb{F}$ be an infinite field of characteristic 2. Let $V$ be an $n$-dimensional vector space over $\mathbb{F}$, and for every $x\in V(\mathcal{H})$, choose a vector $v_x\in V$ such that these vectors are in general position, i.e., any $n$ of them are linearly independent. For every $r$-tuple $A\subset V(\mathcal{H})$, let $w(A)=\bigwedge_{x\in A}v_x\in\bigwedge^{r}V$.  Also, let $e_1,\dots,e_{n}$ be a basis of $V$.
	
	For $i\in [z]$, let $A_{i,1},\dots,A_{i,t}$ be the edges of color $i$. Define the $t$-dimensional tensor $T:[z]^t\rightarrow\mathbb{F}$ as follows. With slight abuse of notation, let 
	$$T(i_1,\dots,i_t)=w(A_{i_1,1})\wedge\dots\wedge w(A_{i_t,t}).$$
	To be more precise, $w(A_{i_1,1})\wedge\dots\wedge w(A_{i_t,t})\in \bigwedge^{n}V$. But each $a\in\bigwedge^{n}V$ can be written as $a=\lambda (e_1\wedge\dots\wedge e_{n})$ for some $\lambda\in\mathbb{F}$, so we can identify $a$ with this $\lambda$.
	
	Note that as $\{v_x\}_{x\in V(\mathcal{H})}$ are in general position, $T(i_1,\dots,i_t)\neq 0$ if and only if $A_{i_1,1},\dots,A_{i_t,t}$ are pairwise disjoint. But then, as $\mathcal{H}$ contains no rainbow matching of size $t$, we get that $T$ is semi-diagonal. By Theorem \ref{thm:main}, this gives $$\mfrank(T)\geq \frac{z}{t-1}.$$
	
	We finish the proof by showing that $\frank_{\ell}(T)\leq \binom{n}{r}$ for every $\ell\in [t]$. For ease of notation, we show this for $\ell=1$, the other cases follow by the same argument. For $I\subset [n]$, let $e_I=\bigwedge_{i\in I}e_i$. Then $\{e_{I}\}_{I\in [n]^{(r)}}$ is a basis of $\bigwedge^{r}V$, so for every $A\subset V(\mathcal{H})$ and $I\in [n]^{(r)}$ there exists $\lambda(A,I)$ such that $$w(A)=\sum_{I\in [n]^{(r)}}\lambda(A,I)e_{I}.$$ But then
	$$T(i_1,\dots,i_t)=\bigwedge_{j=1}^{t}\left(\sum_{I\in [n]^{(r)}}\lambda(A_{i_j,j},I)e_{I}\right)=\sum_{\substack{I_1,\dots,I_t\in [n]^{(r)}\\ I_1\cup\dots\cup I_t=[n]}}\lambda(A_{i_1,1},I_1)\dots\lambda (A_{i_t,t},I_t).$$
	For $I\in [n]^{(r)}$, define the functions $f_{I}:[z]\rightarrow \mathbb{F}$ and $g_I:[z]^{t-1}\rightarrow \mathbb{F}$ as follows. Let
	$$f_{I}(i_1)=\lambda(A_{i_1,1},I),$$
	and
	$$g_{I}(i_2,\dots,i_t)=\sum_{\substack{I_2,\dots,I_t\in [n]^{(r)}\\ I_2\cup\dots\cup I_t=[n]\setminus I}}\lambda(A_{i_2,2},I_2)\dots\lambda (A_{i_t,t},I_t).$$
	Then $$T(i_1,\dots,i_t)=\sum_{I\in [n]^{(r)}}f_{I}(i_1)g_{I}(i_2,\dots,i_t),$$ which shows that $\frank_1(T)\leq \binom{n}{r}$. As this holds for $\frank_{\ell}(T)$ as well for every $\ell\in [t]$, we get $$\mfrank(T)\leq \binom{n}{r}.$$ 
		Comparing the lower and upper bound on the max-flattening rank of $T$, we deduce that $z\leq (t-1)\binom{n}{t}$, finishing the proof.
\end{proof}

In particular, one can slightly modify our proof to show the following extension of the Bollob\'as set pair inequality \cite{B65}, which might be of independent interest.

\begin{theorem}
Let $\mathcal{A}$ be a family of $t$-tuples of subsets of some base set $X$ such that $|A(i)|=r_i$ for every $A\in\mathcal{A}$ and $i\in [t]$. Suppose that $A(1),\dots,A(t)$ are pairwise disjoint for every $A\in\mathcal{A}$, but  $A_1(1),\dots,A_t(t)$ are not pairwise disjoint if $A_1,\dots,A_t\in\mathcal{A}$ are all distinct. Then $$|\mathcal{A}|\leq (t-1)\max_{i\in [t]}\binom{r_1+\dots+r_t}{r_i}.$$
\end{theorem}

\vspace{0.3cm}
\noindent	
{\bf Acknowledgements.} All authors were supported by the SNSF grant 200021\_196965. Istv\'an Tomon also acknowledges the support of Russian Government in the framework of MegaGrant no 075-15-2019-1926, and the support of MIPT Moscow.

\end{document}